\documentclass[12pt]{amsart}

\usepackage[matrix,arrow]{xy}

\usepackage{amssymb,latexsym,amscd,amsmath,amsthm}
\usepackage[latin1]{inputenc}
\usepackage{multind}

\input{xy}

\xyoption{all}

\newcommand{\ds}{\displaystyle}

\newcommand{\gfr}{\mathfrak{g}}
\newcommand{\hfr}{\mathfrak{h}}

\newcommand{\Prym}{\mathrm{Prym}}

\newcommand{\Sym}{\mathrm{Sym}}
\newcommand{\Pic}{\mathrm{Pic}}
\newcommand{\Jac}{\mathrm{Jac}}
\newcommand{\Stab}{\mathrm{Stab}}
\newcommand{\Div}{\mathrm{Div}}
\newcommand{\Aut}{\mathrm{Aut}}

\newcommand{\tr}{\mathrm{tr}}

\newcommand{\ev}{\mathrm{ev}}
\newcommand{\lra}{\longrightarrow}

\newcommand{\hra}{\hookrightarrow}
\newcommand{\sr}{\stackrel}

\newcommand{\ra}{\rightarrow}
\newcommand{\ol}{\overline}
\newcommand{\wh}{\widehat}

\newcommand{\la}{\lambda}
\newcommand{\vp}{\varpi}

\newcommand{\ms}{\mapsto}
\newcommand{\evt}{\tilde{\ev}}
\newcommand{\ul}{\underline}

\newcommand{\ZZ}{\mathbb{Z}}

\newcommand{\QQ}{\mathbb{Q}}

\newcommand{\CC}{\mathbb{C}}

\newcommand{\End}{\mathrm{End}}
\newcommand{\Hom}{\mathrm{Hom}}
\newcommand{\Ind}{\mathrm{Ind}}

\newcommand{\Res}{\mathrm{Res}}

\newcommand{\GL}{\mathrm{GL}}
\newcommand{\SL}{\mathrm{SL}}

\newcommand{\Nm}{\mathrm{Nm}}

\newcommand{\im}{\mathrm{im}}

\newcommand{\Pdon}{\Prym(\pi,\Lambda)}

\theoremstyle{plain}

\newtheorem{thm}{Theorem}[section]

\newtheorem{lem}[thm]{Lemma}

\newtheorem{prop}[thm]{Proposition}

\newtheorem{cor}[thm]{Corollary}

\newtheorem{rem}[thm]{Remark}
\newtheorem{defi}[thm]{Definition}

\newtheorem{ex}{Example}

\begin{document}

\begin{center}
%TITLE
{\large\bf Prym Subvarieties $P_\la$ of Jacobians via Schur correspondances between curves}
\\

%AUTHOR
 Yashonidhi PANDEY\\

%%address
{Chennai Mathematical Institute, \\
Plot No H1, Sipcot IT Park, Padur Post Office,\\
Siruseri 603103, India \\

\texttt{Email: ypandey@cmi.ac.in}}
\end{center}

%%\maketitle
{\footnotesize
%\begin{abstract}
{\bf Abstract:\quad} 
Let $\pi : Z \ra X$ be Galois cover of smooth projective curves with Galois group $W$ a Weyl group of a simple Lie group $G$. For a dominant weight $\la$, we consider the intermediate curve $Y_\la= Z/\Stab(\la)$. One can realise a Prym variety $P_\la \subset \Jac(Y_\la)$ and we denote $\varphi_\la$ the restriction of the principal polarisation of $\Jac(Y_\la)$ upon $P_\la$. For two dominant weights $\la$ and $\mu$, we construct a correspondence $\Delta_{\la \mu}$ on $Y_\la \times Y_\mu$ and calculate the pull-back of $\varphi_\mu$ by $\Delta_{\la \mu}$ in terms of $\varphi_\la$.

{\bf R\'esum\'e:} \quad Soit $\pi: Z \ra X$ un rev\^etement Galoisien de courbes projectives lisses de groupes de Galois $W$ un groupe de Weyl d'une groupe de Lie $G$. Pour un poids dominant $\la$, on consid\`ere la courbe intermediare $Y_\la=Z/\Stab(\la)$. On r\'ealise la vari\'et\'e de Prym $P_\la \subset \Jac(Y_\la)$ et on note $\varphi_\la$ la restriction de la polarisation principale du $\Jac(Y_\la)$ \`a $P_\la$. Pour deux poids dominants $\la$ et $\mu$, on construit une correspondance $\Delta_{\la \mu}$ sur le produit des courbes $Y_\la \times Y_\mu$. On calcule le pull-back de $\varphi_\mu$ par $\Delta_{\la \mu}$ en termes de $\varphi_\la$.
}

\section{Introduction}

Let $G$ be a simple Lie group of type $A$, $D$, or $E$. Let us begin by recalling some constructions done in [M\'erindol] \cite{merindol}, [Kanev] \cite{kanev} \cite{kanev1}, [Lange-Pauly] \cite{lp} and [Lange-Recillas] \cite{lr}. We consider a Galois cover of smooth projective curves $\pi : Z \ra X$ with Galois group the Weyl group $W$ of $G$. To a dominant weight $\la$ of $G$, one can associate an integral, symmetric correspondence 
$$K_\la \subset Y \times Y,$$
known as the Kanev correspondence on the curve $Y = Z/\Stab(\la)$. This correspondence satisfies a relation with the Schur correspondence $S_\la$ defined using the Schur projector in the ring $\End(\Lambda \otimes_{\ZZ} \QQ)$. We study the Prym subvariety of $\Jac(Y)$ defined as 
$$ P_\la = \im (S_\la) \subset \Jac(Y).$$
It can also be realised as $\im(K_\la)|_{\Prym(Y/X)}$ and is isogenous to the Donagi-Prym variety $\Pdon_{\eta}$. Recall that for two dominant weights $\la$ and $\mu$ the Prym varieties $P_{\la}$ and $P_{\mu}$ are known to be isogenous - see [M\'erindol] \cite{merindol} - and in section 10.1 we construct an explicit correspondence $\Delta_{\la \mu}$ giving explicit isogenies 
$$\Delta_{\la, \mu} : P_\la \ra P_{\mu} ; \qquad \Delta_{\mu,\la} : P_{\mu} \ra P_{\la}$$
using the galois cover
$$
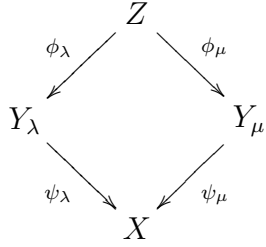
\begin{matrix}
\xymatrix{
  & Z \ar[dl]_{\phi_\la} \ar[dr]^{\phi_\mu} & \\
Y_\la \ar[dr]_{\psi_\la} & & Y_\mu \ar[dl]^{\psi_\mu} \\
                  &   X  &
}
\end{matrix}$$

Actually the construction of the correspondence as well as other results of this paper are all valid in a more general setting which we now describe. Recall a theorem of Springer \cite{springer} which says that any irreducible representation $V$ of a Weyl group $W$ is induced from a rational irreducible representation. Let $\Lambda \subset V$ be some lattice invariant under the action of the Weyl group. Indeed the most important example is that of the Cartan subalgebra $\hfr$ of the Lie group corresponding to $W$and of the weight lattice $\Lambda \subset \hfr$. Let $\lambda \in \Lambda$ be any element. Then in the situation of a Galois \'etale cover $\pi: Z \ra X$, we can again consider for $\Stab(\lambda) \subset W$ an intemediate curve $Y=Z/\Stab(\lambda)$ and for two such curves $Y_\la$ and $Y_\mu$ we show the construction of the correspondence. To use suggestive notation, we use $\Lambda$ for the weight lattice and call $\la$ and $\mu$ as weights.

We also show the equalities (Corrolary \ref{isogeniedelta})
$$ \Delta_{\mu,\la} \Delta_{\la,\mu} = [N]_{P_{\la}}  ; \qquad \Delta_{\la,\mu} \Delta_{\mu,\la} = [N]_{P_{\mu}},$$
where $[N]$ denotes the  multiplication by the integer $$\frac{|W|^2 (\la, \la,) (\mu, \mu)}{|\Stab(\la)| |\Stab(\mu)| (\dim(V))^2}.$$ In the context of abelianisation one takes $V$ as $\hfr$-the Cartan sub-algebra. 

In section 3, we define the Prym varieties $P_V$ and $P_\la$. We show in proposition \ref{polonisotypcompsplits} that the pull-back of the principal polarisation on $\Jac(Z)$ to 
$\prod_{V \in \chi(W)} P_V$ splits. In proposition \ref{polonPlasplits} we show that the isotypical component $P_V$ is isogenous to a product of Prym varieties $P_\la$ and that the pull-back of the polarisation splits on the product $\prod_{i=1}^n P_{\la_i}$ for any choice of an orthogonal basis $\la_1, \cdots, \la_n$ of $V$. 

In section 4 we define a rational correspondence $S_{\la, \mu}$ on $Y_\la \times Y_\mu$ as 
$$\frac{1}{|\Stab(\la) \Stab(\mu)|}(\phi_\la \times \phi_\mu)_*(S_\la \circ S_\mu).$$ It can be described as follows 
$$\begin{matrix}
S_{\la,\mu}: & Y_\la & \ra & \Div_{\QQ}(Y_\mu) \\
         & y_1       & \ms & \sum_{i=1, \cdots d} (\mu g_i, \la) \phi_\mu(z^{g_i})
\end{matrix}$$
where $z \in Z$ satisfying $\phi_\la(z) = y_1$ and $\{g_1, \cdots, g_d \}$ 
is a system of right coset representatives of $\Stab(\mu)$ in $W$ and $W$ acts upon the curve $Z$ by monodromy.

Notice that if $\la_2= \la_1$, this correspondence ressembles the Schur correspondence $\ol{S_\la}$ defined in [Lange-Pauly]\cite{lp}. In that paper, $\ol{S_{\la}}$ is related to the Kanev correspondence [Kanev]\cite{kanev} by theorem 3.5
$$\ol{S_\la} = |H|^2 (\ol{K_\la} - \ol{\Delta} + ((\la,\la)+1) \ol{T}). \label{relation}$$ 

In \cite{lk} the authors Lange and Kanev remark that the generalisation of the Kanev correspondence denoted $\Delta_{\la \mu}$ for two weights $\la
$ and $\mu$ over an arbitrary curve in immediate. It is integral. Then we relate the two correspondences in theorem \ref{deltaklien} by the relation

$$\Delta_{\la, \mu} = S_{\la, \mu} + r \ol{T} $$
where $\ol{T}$ is the trace correspondence and for some $r \in \QQ$. In perticular, $S_{\la, \mu}$ and $\Delta_{\la, \mu}$ induce the same isogeny from $P_\la$ to $P_\mu$. As the calculations with $S_{\la, \mu}$ employing Schur type relations are easier, we obtain the following formulae for the exponent of the isogeny

$$ \Delta_{\mu, \la} \Delta_{\la, \mu} = K_{\mu, \la} K_{\la, \mu} = N$$
by the proposition \ref{k21k12} and the corollay \ref{isogeniedelta}. The corollay \ref{isogeniedelta} simplifies the proof of theorem 6.5 in [Kanev] \cite{kanev1}.

Let $\varphi_\mu$ denote the principal polarisation on $\Jac(Y_\mu)$. In section 6, We show in Corrolary \ref{compresult} that the pull-back of $\varphi_\mu$ by $\Delta_{\la, \mu}$ is equal to $[N] \varphi_\la$. Then in section 7, we calculate the integer $[N]$ for all couples of fundamental weights for all Lie Algebras.
In section 8, we use the correspondence $\Delta_{\la, \mu}$ to give precise descriptions of some isogenies in the context of Abelianisation. More precisely, let $H^1(Z, \ul{T})^W$ denote $\ul{T}-$bundles on $Z$ that are $W-$invariant for the twisted action. Associated to a dominant weight $\la$, we have an evaluation map $\ev_\la$ with values in $\Jac(Z)$ that associates to $T-$bundles the line bundle obtained by weight $\la$ acting as character of $T$. The line bundles in the image of $\ev_\la$ can be endowed with a canonical $\Stab(\la)$ linearisation and thus descend to the intermediate curve $Y_\la$. Consider the diagram,
\begin{equation*}
\xymatrix{
H^1(Z,\ul{T})^W \ar[r]^{~~~\evt_\la} \ar[rd]_{\evt_\mu} & P_\la \ar[d]^{\Delta_{\la, \mu}}\ar[r] & \Jac(Y_\la) \ar[d]^{\Delta_{\la, \mu}}\\
                                         & P_\mu \ar[r] & \Jac(Y_\mu)
}
\end{equation*}
We prove in proposition \ref{pdonpla} the equality
$$ \Delta_{\la, \mu} \tilde{\ev_\la} = {\frac{|W|}{\dim(V)}} \frac{(\la, \la) (\mu, \mu)}{|H_\la||H_\mu|} \tilde{\ev_\mu}.$$
Moreover, we construct an inverse isogeny 
$$\delta: P_\la \ra H^1(Z, \ul{T})^W$$
to the isogeny $\evt_\la$ of \ref{pdonpla}. We obtain
$$\delta \evt_\la=[M],$$
where $M$ is the exponent of the group $P(G)/\ZZ[W]\la$.

I wish to thank my advisor Christian Pauly for his advices and help in the preperation of this paper and my thesis. 

\section{Correspondences between algebraic curves}
\begin{defi} A correspondence $D$ between two algebraic curves $C_1$ et $C_2$ is a divisor $D$ on the product $C_1 \times C_2$.  Two correspondences $D$ et $D'$ between $C_1$ et $C_2$ are equivalent if there exists divisors $D_1$ and $D_2$ on $C_1$ and $C_2$ respectively such that  $$D'= D + D_1 \times C_2 + C_1 \times D_2.$$
\end{defi}

To a correspondence $D$ we denote by $\gamma_D: \Jac(C_1) \ra \Jac(C_2)$ the associated map and by $\gamma'_D$ the dual isogeny obtained by the Rosati involution. We denote by $Corr(C_1,C_2)$ the isomorphism classes of correspondences between $C_1$ and $C_2$.

\begin{thm} [Theorem 11.5.1 Birkenhake-Lange \cite{bl}] The map $D \ms \gamma_D$ induces an isomorphism
$Corr(C_1, C_2) \ra  \Hom(\Jac(C_1), \Jac(C_2))$.
\end{thm}

In the case, $C_1 = C_2 =C$, this theorem allows to translate the ring structure and the Rosati involution on $\End(\Jac(C))$ to $Corr(C)$. We denote by $\tau: C \times C \ra C \times C$ the map which switches the two factors.

\begin{prop} Let $D$ be a correspondence on a curve $C$. We have $$\gamma_{\tau^*D} = \gamma'_D.$$
\end{prop}

In perticular, if the correspondence $D$ is symmetric, then the endomorphism  $\gamma_D$ is symmetric with respect to the Rosati involution.

\section{Schur projectors, abelian subvarities of Jacobian and  decomposition}
Let $W$ be an arbitrary finite group. We denote by $\CC[W]$ the group algebra associated to the  group $W$. We have a decomposition of algebras 
\begin{equation}
\CC[W] = \ds{\oplus_{\omega \in \chi(W)}} \End(V_{\omega}) \label{algebradecomp} 
\end{equation}
where $V_{\omega}$ is the complex irreducible representation corresponding to characters $\omega$ et $\chi(W)$ is the set of irreducible characters de $W$.

Following \cite{merindol}, we denote by $(C)$ the condition upon a finite group $W$ that all its  irreducible representations are absolutely irreducible. Recall absolute irreducibility means that all the irreducible representations can be obtainted after extension of scalars from a rational irreducible representation. The Weyl group satisfies this condition according to a theorem of Springer \cite{springer}. We thus have,

\begin{prop} \label{prop(C)} Let $V$ be an irreducible representation of a group $W$ satisfying the condition $(C)$ over a field of $k$ of characteristic zero. We have a canonical isomorphims upto scalars of $V$ with its dual $V^*$. Equivalently we have a non-degenerate symmetric $G-$invariant bilinear form on $V \times V \ra k$ unique upto scalars. 
\end{prop}
\begin{proof} By the condition $(C)$ we may assume that the field $k$ is the rationals $\QQ$ and that $V$ is defined over $\QQ$. Now the character of $V^*$, the dual of $V$, is $\ol{\chi_V}$ where $\chi_V$ denotes the character of $V$. Since $V$ is defined over $\QQ$, therefore the character takes values in $\QQ$ and thus $\chi_{V^*}= \ol{\chi_V} = \chi_V$. Thus $V$ is isomorphic to $V^*$. Since it is irreducible, thus upto scalars there is only one isomorphism by Schur's lemma. The existence of a bilinear form satisfying the said properties is standard.
\end{proof}

\begin{rem} The Cartan-Subalgebra $\hfr$ of a semi-simple simple group $G$ is also an irreducible representation of the Weyl group $W$ of $G$. In this paper, we shall often consider $\hfr$ as being induced from a rational irreducible representation of $W$. Moreover the bilinear form of the above proposition is the celebrated Cartan-Killing form. In perticular, for a Weyl groups any of its irreducible representation can be endowed with a symmetric non-degenerate $W-$invariant bilinear form and is isomorphic to its dual representation. However, the bilinear form may not be positive or negative definate unlike the Cartan-Killing form.
\end{rem}
 
Let $V$ be an irreducible representation of $W$ defined over $\QQ$. 

\begin{defi} The  \emph{Schur projector} is the element
$$\tilde{p_V} = \ds{\sum_{g \in W}} \tr_V(g)g \in \QQ[W]$$
where $\tr_V(g)$ denotes the trace endomorphism of $V$ induced from multiplication by $g$.
\end{defi}

The element $\tilde{p_V}$ satisfies the relation
$$\tilde{p_V}^2 = \frac{|W|}{\dim(V)} \tilde{p_V}.$$
Therefore, $p_V= \frac{\dim(V)}{|W|} \tilde{p_V}$ is an idempotent in $\CC[W]$.
The action of $p_V$ on a $W-$module is that of projection upon the isotypical component of $V$.

 Let $(,):V \times V \ra \CC$ be a $W-$invariant bilinear form on $V$. It is unique upto scalars. 

For $\lambda \in V \setminus \{0\}$, we denote $$h_\lambda= \sum_{g \in W}(\lambda g , \lambda) g.$$ According to Lange-Recillas \cite{lr} the element $H_\la= \frac{\dim(V)}{|W|(\la,\la)} h_\la$ is an idempotent of $\CC[W]$. The action of $H_\la$ of $V$ is that of projection on the line generated by $\la$ in $V$.

Let $V$ admit $(\lambda_1, \cdots, \lambda_d)$ as a basis. We get the formulae
\begin{equation}
p_V = H_{\la_1} + \cdots + H_{\la_d}. \label{decompbasis}
\end{equation}

\begin{thm} \cite{bl} Let $(A,L)$ be a polarised abelian variety. There is a bijective correspondence between the sub abelian varieties of $A$ and the idempotents in the ring $\End(A) \otimes_{\ZZ} \QQ$ that are symmetric with respect to the Rosati involution. The bijection is given by associating to an idempotent $u$, the subvariety $\im (u)$ and to an abelian subvariety $B$ we associate the composition of the maps
$$A \stackrel{\phi_L}{\ra} \hat{A} \stackrel{\hat{i_B}}{\ra} \hat{B} \stackrel{\psi_L}{\ra} B \stackrel{i_B}{\ra} A.$$
where $i_B$ denotes the inclusion of $B$ in $A$ and $\psi_{L|B}= q \phi_L^{-1}$ denotes the isogeny dual to $\phi_{L|B}$.
\end{thm}

We shall use the above theorem to define subvarieties of $\Jac(Z)$ when the curve $Z$ admits an action by a group $W$ satisfying the hypothesis $(C)$.

Let $\rho: W \ra \GL(V)$ be a rational representation , that is  defined over $\QQ$. The Schur projector $p_V$ induces an element denoted $\rho_V$ in $\End_{\QQ}(\Jac(Z))$ by the action of $W$ on the curve $Z$. 

\begin{defi} We denote $P_V= \im(\tilde{p_V})$ and $P_\lambda= \im(H_\lambda)$ the sub abelian varieties of the Jacobian of $Z$.
\end{defi}

The decomposition of the group algebra $\CC[W]$ (\ref{algebradecomp}) and that of the Schur projector (\ref{decompbasis}) gives us the following theorem.

\begin{thm} [\cite{merindol} Thm 2.5] Let $W$ be a finite group satisfying the condition $(C)$. The abelian varieties $\Jac(Z)$ and $\prod_{V \in \chi(W)} P_V$ are isogenous. Moreover, ([\cite{merindol} Prop 4.5]) the abelian variety $P_V$ is isogenous to the product $[P_\lambda]^{dim(V)}$.
\end{thm}

More generally, we have the \emph{ isotypical decomposition } of an abelian  variety $A$ unique upto permutation of factors (cf. \cite{lr} Proposition 1.1) and 
decomposition in terms of \emph{generalised Prym varities} or Prym-Donagi varities (cf. \cite{donagi} Formula 5.3).

The group $W$ acts on the curve $Z$ on the right. Thus it can act on the $\Jac(Z)$ as $(g, D) \ms Dg$ where $D$ denotes a divisor in $\Jac(Z)$. But we rather choose the action to be given by pull-back of line bundles on $Z$
$$\begin{matrix}
G \times \Jac(Z) & \ra & \Jac(Z) \\
(g,D)            & \ms & Dg^{-1}
\end{matrix}$$
Notice that pull-back will be an action on the right. 

\begin{prop} \label{isodual} Let $Z$ be a smooth projective curve. Let $g \in \Aut(Z)$, and let us also denote by $g$ the automorphism on $\Jac(Z)$ induced by $g$. Then $g$ preserves the principal polarisation on $\Jac(Z)$. The following diagram commutes
$$\begin{matrix}
\xymatrix{
\Jac(Z) \ar[r] \ar[d]^{g^{-1}} & \wh{\Jac(Z)} \ar[d]^{\wh{g}} \\
\Jac(Z) \ar[r]               & \wh{\Jac(Z)}
}
\end{matrix}$$
\end{prop}
\begin{proof} The image of the Abel-Jacobi map $\Sym^{g_Z-1}(Z) \ra \Pic^{g_Z-1}(Z)$ is  preserved by $g$. This image is the Riemann theta divisor $\Theta$ which induces  the principal polarisationon $\Jac(Z)$. Let $\Theta \subset \Jac(Z)$ be a divisor representing the principal polarisation on $\Jac(Z)$. Then $g^* \Theta$ is the translate $T_{\alpha}^* \Theta$ of $\Theta$ for a certain $\alpha \in \Jac(Z)$. The isomorphism between $\Jac(Z)$ and $\widehat{\Jac(Z)}$ given by $\Theta$ is
$$\begin{matrix}
\Theta : & \Jac(Z) & \ra & \widehat{\Jac(Z)} \\
         & z       & \ms & T_z^* \Theta - \Theta.
\end{matrix}$$
So let $g \in \Aut(Z)$. Let us compute the composition
$$\Jac(Z) \ra^{\Theta} \wh{\Jac(Z)} \sr{\wh{g}}{\ra} \wh{\Jac(Z)} \sr{{\Theta}^{-1}}{\ra} \Jac(Z).$$
Now $g$ seen as an element of $\Aut(\Jac(Z))$ transforms the divisor $T_x^*\Theta - \Theta$ into $g^*(T_x^*\Theta - \Theta)$. Now we have $$ g^*(T_x^*\Theta - \Theta)= (T_x g)^* \Theta - g^* \Theta = \{g T_{g^{-1}(x)}\}^* \Theta - g^* \Theta = T^*_{g^{-1}(x)} g^* \Theta - g^* \Theta.$$ By the equality $g^* \Theta = T_{\alpha}^* \Theta$, the expression becomes $T^*_{g^{-1}(x)} T^*_{\alpha} \Theta - T^*_{\alpha} \Theta$ which is equal to $T^*_{g^{-1}(x)} \Theta  -  \Theta$ by the theorem of square. Thus, the diagram commutes.

\end{proof}

\begin{prop} \label{polonisotypcompsplits} Let $Z$ be a curve admitting an action by a Weyl group $W$. Consider the decomposition
$$\prod_{V \in \chi(W) } P_V \ra \Jac(Z).$$
where $\chi(W)$ denotes the irreducible representations of $W$. The pull-back of the principal polarisation on $\Jac(Z)$ splits.
\end{prop}
\begin{proof}
For an irreducible representation $V$, the inclusion $i_V: P_V \ra \Jac(Z)$ is the same as multiplication by $\chi_V= \frac{|\dim(V)|}{|
G|} \sum_{g \in W} \tr(g) g$ and thus 
$$\wh{i_V} = \sum_{g \in W} \tr(g) \wh{g}.$$
For two irreducible representations $U$ and $V$ of $W$, let us compute the composition 
$$\Jac(Z) \sr{\chi_V}{\lra} P_V \sr{i_V}{\lra} \Jac(Z) \sr{\Theta}{\lra} \wh{\Jac(Z)} \sr{\wh{i_U}}{\lra} \wh{P_U}.$$ 
Let $p \in P_V$. It is mapped to $\Theta(p)= [T_p^* \Theta - \Theta]$ which goes to $\wh{\chi_U}[T_p^* \Theta - \Theta]$. Now $\wh{\chi_U}= \sum_{g \in W} tr(g) \wh{g}$. By linearity, this sum by th proposition \ref{isodual} is equal to
$$[T^*_{x} \Theta - \Theta],$$
where $x= \sum_{g \in W} \tr_U(g) p g^{-1}$. As irreducible representations of  Weyl groups are self dual we have the equalities 
$$\tr_U(g^{-1})=\tr_{U^*}(g)= \ol{\tr_U(g)} = \tr_U(g),$$
so $x= \chi_U(p)$.
Now let $p = \chi_V(z)$. Thus $x = \chi_U \chi_V(z)$, but $\chi_U \chi_V=0 \in \CC[W]$ as $\chi_V$ and $\chi_U$ project on the isotypical component of $V$ and $U$ respectively. Thus the pull-back of the polarisation on $\Jac(Z)$ splits.
\end{proof}

\begin{prop} For any simple Lie algebra $\gfr$, we can find weights $v_i$ perpendicular to eachother with respec to the Cartan-Killing form and forming a basis for the Cartan-subalgebra $\hfr$ over $\CC$.
\end{prop}
\begin{proof} We can take for $w_1$ the first fundamental weight $\varpi_1$. For $w_2$ we can take $\varpi_2 - \frac{(\varpi_2, \varpi_1)}{(\varpi_1, \varpi_1)} \varpi_1$ and so on. Then we clear the denominators and call these vectors $v_i$.
\end{proof}

\begin{rem} An analogus statement can be made for any irreducible representation of the Weyl group.
\end{rem}

\begin{prop} \label{polonPlasplits} Let $Z$ be a curve admitting an action by a Weyl group $W$. Let $V$ be an irreducible representation of $V$ defined over $\QQ$ and let $\{v_i\}_{i=1, \cdots, n}$ form an orthogonal basis of $V$. Consider the morphism from 
$$\prod_{i=1,\cdots,n} P_{v_i} \ra \Jac(Z).$$
The pull-back of the principal polarisation on $\Jac(Z)$ to the product of Prym varieties $P_{v_i}$ splits.
\end{prop}
\begin{proof}
The inclusion of $i_v: P_v \in \Jac(Z)$ is the same as $S_v/q_v$ where $q_v$ denotes the exponent of $S_v$. Thus we also have $$\wh{i_v}= \wh{S_v/q_v}.$$ So let us compute the composition
$$\Jac(Z) \sr{S_{v_1}}{\lra} P_{v_1} \sr{S_{v_1}/q_1}{\lra} \Jac(Z) \sr{\Theta}{\lra} \wh{\Jac(Z)} \sr{\wh{S_{v_2}/q_2}}{\lra} \wh{P_{v_2}}.$$
Let $p \in P_{v_1}$. It is mapped to $\Theta(p)= [T_p^* \Theta - \Theta]$ which goes to $\wh{S_{v_2}}[T_p^* \Theta - \Theta]$. Now $\wh{S_{v_2}}= \sum_{g \in W} (v_2 g, v_2) \wh{g}$. Thus the sum becmes $$ \sum_{g \in W}(v_2 g, v_2) \wh{g}[ T_p^* \Theta - \Theta].$$ By linearity and the theorem of square, this sum by the proposition \ref{isodual} is equal to 
$$[T^*_{x} \Theta - \Theta],$$
where $x=\sum_{g \in W} (v_2 g, v_2) p g^{-1}$.
Now by the symmetry of the Cartan-Killing form we have $$x= \sum_{g \in W} (v_2 g, v_2) p g^{-1} = \sum_{g \in W} (v_2 , v_2 g^{-1}) p g^{-1} = S_{v_2} p.$$ Let $p = S_{v_1}(z)$. Then by lemma \ref{sl2sl1} and the orthogonality of $v_1$ and $v_2$ we get $S_{v_2}p=0$. Thus the composition of arrows is zero, that is the polarisation splits.
\end{proof}

\section{Schur correspondence for two dominant weights  $\lambda_1$ et $\lambda_2$ }

We consider the following situation. Let $\pi: Z \ra X$ be a Galois cover of smooth projective curves with Galois group $W$. Let $\lambda_1$ and $\lambda_2$ be two dominant weights. We write $H_i = \Stab(\lambda_i)$ in $W$ and $Y_i=Z/H_i$. Let $$\rho: W \ra \GL_{\QQ}(V)$$ be an irreducible  representation of $W$ on $\QQ$ satisfying the condition $(C)$ [cf. \cite{merindol}] that all its irreducible representations are absolutely irreducible. Let $(,)$ be a bilinear symmetric negative definate form on $V$. In the following we consider right actions of $W$ on $Z$ and $V$. Moreover, the action of the fundamental group $\pi_1(X)$ on the curve $Z$ and $Y_i$ by monodromy  is on the right. 

For $i=1,2$ we introduce the correspondences on $Z$

$$\begin{matrix} 
S_{\lambda_i}: & \Jac(Z) & \ra & \Jac(Z) \\
               & z       & \ms & \ds{\sum_{g \in W}} (\lambda_i g, \lambda_i) z^g.
\end{matrix}$$

We have $S_{\lambda_i} \in \End_{\QQ}(\Jac(Z))$. We denote $P_{\lambda_i} = \im(S_{\lambda_i})$. According to [\cite{lr}], we have 
\begin{equation}
S^2_{\lambda_i} = \frac{|W| (\lambda_i, \lambda_i)}{\dim(V)^2} S_{\lambda_i}.
\end{equation}

We now turn to a useful technical lemma.
\begin{lem}[Merindol, \cite{merindol}] \label{4-linear} Let $W$ be a group satisfying the hypothesis $(C)$. Let $V$ be an irreducible representation of $W$ and denote by $(,)$ a symmetric non-degenerate $W-$invariant bilinear form on $V$ unique upto scalars. Let $F$ be a form 
$$F: V \times V \times V \times V \ra \QQ$$
satisfying
$$F(a w, b, c, dw) = F(a, bw, cw, d) = F(a, b, c,d)$$
for all $a,b,c,d \in V $ and $h \in W$. Then there exists a constant $q \in \QQ$ such that we have for all $(a,b,c,d) \in V^4$,
$$ F(a,b,c,d) = q (a,d) (c,d).$$
\end{lem}

\begin{lem} \label{sumcal} Let $V$ be an irrducible representation of a group $W$ satisfying condition $(C)$. Suppose that $V$ admits a positive-definate bilinear form $(,)$. For $\la_i \in V$, where $i =1,..,4$, define $B(\la_1,\la_2,\la_3, \la_4) = \sum_{h \in W} (\la_1 h , \la_2) (\la_3, \la_4 h)$. We have the equality $$B(\la_1,\la_2,\la_3,\la_4) = \frac{|W|}{\dim(V)^2} (\la_1, \la_4) (\la_2, \la_3).$$
\end{lem}
\begin{proof}
Notice that such sums satisfy the conditions of lemma \ref{4-linear}. So there exists a constant $q \in \QQ$ such that the above sum equals $q (\la_1, \la_4) (la_2,\la_3)$. To calculate the constant we may take an orthonormal basis $\mu_1, \cdots, \mu_n$ of $V$ since the form is positive-definate. Expressing the bilinear form and the matrix $m(h)$ of $h \in W$ in terms of the orthonormal basis, we obtain
$$B(\mu_i, \mu_j, \mu_j, \mu_i) = \sum_{h \in W} m(h)_{i,i} m(h)_{j,j} = q (\mu_i, \mu_i)(\mu_j,\mu_j) = q.$$
Summing over all couples of basis vectors, we get
$$\dim(V)^2 q = \sum_{i,j} \sum_{h \in W} m(h)_{i,i} m(h)_{j,j} = \sum_{h \in W} \chi_V(h)^2 = |W|.$$
We thus get the value of $q$.
\end{proof}

\begin{lem} \label{sl2sl1} We have the equality $$S_{\lambda_2} S_{\lambda_1} (z) = \frac{|W| (\lambda_1, \lambda_2)}{\dim(V)^2} \ds{\sum_{t \in W}} (\lambda_1 t, \lambda_2) z^t.$$
\end{lem}
\begin{proof} 
We have,
$$
S_{\lambda_2} S_{\lambda_1}(z) = S_{\lambda_2}(\sum_{g \in W}(\lambda_1 g, \lambda_1) z^g) = \sum_{h \in W} \sum_{g \in W} (\lambda_2 h , \lambda_2) (\lambda_1 g , \lambda_1) z^{gh}$$
We put $t=gh$ and expressing $g=th^{-1}$, the sum becomes
$$\sum_{t \in W} \sum_{h \in W} (\lambda_2 h , \lambda_2) (\lambda_1 t, \lambda_1 h) z^t$$ which by \cite{lr} Corollaire 3.3 (or the lemmas \ref{4-linear} and \ref{sumcal}) is equal to 
$$ \frac{|W| (\lambda_1, \lambda_2)}{\dim(V)^2} \sum_{t \in W} (\lambda_2, \lambda_1 t) z^t$$
\end{proof}

This formulae is not symmetric in $\lambda_1$ and $\lambda_2$.

Let us recall the following geometric configuration 

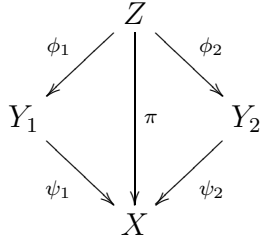
\begin{equation*}
\xymatrix{
            & Z \ar[dd]^{\pi} \ar[ld]_{\phi_1} \ar[rd]^{\phi_2} & \\
Y_1 \ar[rd]_{\psi_1} & & Y_2 \ar[ld]^{\psi_2} \\
              & X &
}
\end{equation*}

\begin{defi} Let $A \in Corr(Z \times Z)$ be a correspondence over $Z$. Let $Y_1$ and $Y_2$ be quotients of $Z$. Then $A$ induces a correpondance $\ol{A}$ over $Y_1 \times Y_2$ defined by $$\ol{A}(y_1) = \sum_{z \in Z, \phi_1(z)=y_1} \phi_2(A(z)).$$ In fact, $\ol{A}: \Jac(Y_1) \ra \Jac(Y_2)$ is defined by the composition of maps $$\ol{A}: \Jac(Y_1) \stackrel{\phi_1^*}{\lra} \Jac(Z) \stackrel{A}{\ra} \Jac(Z) \stackrel{\Nm_{\phi_2}}{\lra} \Jac(Y_2).$$
\end{defi}

\begin{prop} \label{compcorbasepoint} Let $z_0 \in Z$ be a point in the fiber of $\phi_1$ over $y \in Y_1$. Fix a system of left representatives $\{g_i\}_{i=1,\cdots,d}$ of $H_2$ in $W$. We have the equality
$$ \ol{S_{\la_2} \circ S_{\la_1}}(y) = \frac{|W| |H_1| |H_2| (\la_1, \la_2)}{\dim(V)^2} \sum_{i=1,\cdots, d} (\la_1 g_i, \la_2) \phi_2(z_0 g_i). $$
\end{prop}
\begin{proof} We have
$$ \ol{A}(y) = \sum_{z \in \phi_1^{-1}(y)} \phi_2(S_{\la_2} S_{\la_1} (z))$$
By lemma \ref{sl2sl1} we get
$$= \frac{|W|}{\dim(V)^2} (\la_2,\la_1) \sum_{z \in \phi^{-1}_1(y)} \sum_{s \in W} (\la_1 s, \la_2) \phi_2 (z s)$$
and fixing $z_0 \in \phi_1^{-1}(y)$, we get
$$ = \frac{|W|}{\dim(V)^2} (\la_2, \la_1) \sum_{h_1 \in H_1} \sum_{s \in W} (\la_1 s, \la_2) \phi_2(z h_1 s)$$
putting $h_1 s = t$ and $q =\frac{|W|}{\dim(V)^2}$, we get
\[\begin{array}{c} =q(\la_1, \la_2)\sum_{h_1 \in H_1} \sum_{t \in W} (\la_1 h_1^{-1} t, \la_2) \phi_2(z_0 t) \\
= q(\la_1, \la_2) \sum_{t \in W} (\sum_{h_1 \in H_1} \la_1 h_1^{-1}  , \la_2 t^{-1}) \phi_2(z_0 t) \\
= q(\la_1, \la_2) |H_1| \sum_{i=1, \cdots d} \sum_{h_2 \in H_2} (\la_1, \la_2 h_2^{-1} g_i^{-1}) \phi_2(z_0 g_i) \\
= q(\la_1, \la_2) |H_1| \sum_{i=1, \cdots, d} \sum_{h_2 \in H_2} (\la_1 g_i, \la_2 h_2^{-1}) \phi_2 (z_0 g_i) \\
 = q(\la_1, \la_2) |H_1| \sum_{i=1, \cdots, d} (\la_1 g_i, \sum_{h_2 \in H_2} \la_2 h_2^{-1}) \phi_2(z_0 g_i) \\
= q(\la_1, \la_2) |H_1| |H_2| \sum_{i=1, \cdots, d} (\la_1 g_i, \la_2) \phi_2(z_0 g_i).
\end{array}
\] 

\end{proof}

The above proposition motivates the following definition.

\begin{defi} \label{k12} We define the Schur correspondence $S_{\la_2, \la_1}$ for two weights $\la_1$ and $\la_2$ on $Y_1 \times Y_2$ as $$\frac{1}{|H_1||H_2|}(\phi_1 \times \phi_2)_*(\ol{S_{\la_2} \circ S_{\la_1}}).$$

We fix a system of left coset representatives of $H_2$ in $W$  $$W = \ds{\coprod_{i=1}^d} H_2 g_i.$$ The Schur correspondence for two weights $S_{\la_1, \la_2}$ can be described as
$$\begin{matrix}
S_{1,2} : & \Jac(Y_1) & \ra & \Jac(Y_2) \\
          & y_1       & \ms & \ds{\sum_{i=1}^d} (\lambda_1 g_i, \lambda_1) \phi_2(z^{g_i})
\end{matrix}$$
where $z \in Z$ with $\phi_1(z) = y_1$. Here the action of $g$ on $z$ is through monodromy.
\end{defi}

\begin{rem} Since the correspondence $S_{1,2}$ only differs from $\ol{S_{\la_2} S_{\la_1}}$ by a constant, so it is independant of the choices of the point $z$ in the fiber of $\phi_1$ and the system of coset representatives of $H_2$ in $W$ made.
\end{rem}

\begin{rem} For the definition of $S_{2,1}$, one would have to fix a system of right representatives of $H_1$ in $W$ giving the decomposition $W = \ds{\coprod_{i=1}^{d'}} H_1 g_i^{\prime}$. The numbers $d'$ and $d$ of the definition \ref{k12} need not be equal in general for a couple of weights $\lambda_1$ and $\lambda_2$.
\end{rem}

\begin{rem} This proposition generalises for two weights $\lambda_1$ and $\lambda_2$ the proposition 3.1 of \cite{lp}. Indeed, if $\lambda_1 = \lambda_2 = \lambda$, then $\ol{A} = \ol{S_{\lambda}}$ and $S_{1,2}(y_1)= \ds{\sum_{i=1}^d} (\lambda g_i, \lambda) \phi(z^{g_i})$ and $H_1 = H_2 = H$.
\end{rem}

\begin{prop} Let $\phi: Z \ra Y$ be a covering of smooth projective curves. Let $\phi^* : \Jac(Y) \ra \Jac(Z)$ and $\Nm: \Jac(Z) \ra \Jac(Y)$ be the pull-back and norm homomorphisms respectively. We have $$\widehat{\phi^*} = \Nm_{\phi}, \qquad \widehat{\Nm_{\phi}} = \phi^*$$ where $\widehat{\phi^*}$ and $\widehat{\Nm_{\phi}}$  are maps dual to $\phi^*$ and $\Nm$ respectively.
\end{prop}

\begin{proof} It suffices to prove $\widehat{\Nm_{\phi}}= \phi^*$. The other relation follows by taking duals. We have the following commutative diagram
$$\begin{matrix}
\xymatrix{
Z \ar[r]^{\alpha{z_0}~~~~} \ar[d]^{\phi} & \Jac(Z) \ar[d]^{\Nm_{\phi}} \\
Y \ar[r]^{\alpha{y_0}~~~~}               & \Jac(Y) 
}
\end{matrix}$$
where $z_0$ and $y_0$ are such that $\phi(z_0)=y_0$ and
 $$\begin{matrix}
\alpha_{z_0}: & Z & \ra & \Jac(Z) \\
              & z & \ms & z-z_0.
\end{matrix}$$

By applying the functor $\Jac$ we obtain $\wh{\Nm_{\phi}} = \phi^*$ by the following commutative diagram
$$\begin{matrix}
\xymatrix{
\Jac(Z) & \wh{\Jac(Z)} \ar[l]_{\wh{\alpha_{z_0}}}  \\
\Jac(Y) \ar[u]^{\phi^*}  & \wh{\Jac(Y)} \ar[l]^{\wh{\alpha_{y_0}}} \ar[u]_{\wh{\Nm_{\phi}}}  
}
\end{matrix}$$

\end{proof}

\begin{cor} The Schur correspondence $S_{\lambda} \in \End(\Jac(Z))$ is symmetric under the Rosati involution.
\end{cor}
\begin{proof}
We have $\widehat{S_{\lambda}}(z) = \ds{\sum_{g \in W}} (\lambda g, \lambda) z^{g^{-1}}$ which by proposition \ref{isodual} is equal to $$\ds{\sum_{g \in W}}( \lambda, \lambda g^{-1}) z^{g^{-1}} = S_{\lambda}(z).$$
\end{proof}

Let us also denote by $S_{1,2}: \Jac(Y_1) \ra \Jac(Y_2)$ the morphisms between the jacobians induced by the correspondence $S_{1,2}$. Recall that $P_{\la_i} \subset \Jac(Y_i)$ is the image of the endomorphism of $\Jac(Y_i)$ induced by the Schur correspondence $S_{\la_i}$.

\begin{prop} \label{propk12} The morphism $S_{1,2}$ maps $P_{\la_1}$ to $P_{\la_2}$. We also have the relation  $$\wh{S_{1,2}}=S_{2,1}.$$ 
\end{prop}

\begin{proof} We denote by $\tilde{q}$ the constant $\frac{\dim(V)^2}{|H_1||H_2||W|(\lambda_1, \lambda_2)}$.
By proposition \ref{compcorbasepoint}, we have the relations
$$ S_{1,2}= \tilde{q} \ol{S_{\lambda_2} S_{\lambda_1}} \qquad S_{2,1}= \tilde{q} \ol{S_{\lambda_1} S_{\lambda_2}}. $$
We shall omit the constant $\tilde{q}$ to ease notation. So we get $\im(S_{1,2})= \im(\ol{S_{\la_2}} \ol{S_{\la_1}}) \subset \im(S_{\la_2}) = P_{\la_2}.$ Now $\ol{S_{\lambda_2} S_{\lambda_1}}$ 
is by definition the composition of the homomorphisms
\begin{equation} \label{sl1sl2bar} \ol{S_{\lambda_2} S_{\lambda_1}} : \Jac(Y_1) \stackrel{\phi_1^*} {\lra} \Jac(Z) \stackrel{S_{\lambda_1}}{\lra} \Jac(Z) \stackrel{S_{\lambda_2}}{\lra} \Jac(Z) \stackrel{\Nm_{\phi_2}}{\lra} \Jac(Y_2)
\end{equation} Taking duals for the sequence (\ref{sl1sl2bar}) and using the relations $\widehat{\Nm_{\phi_i}}= \phi_i^*$, $\widehat{S_{\lambda_i}}=S_{\lambda_i}$ and $\widehat{\phi^*_i}= \Nm_{\phi_i}$
we obtain that
$$\widehat{\ol{S_{\lambda_2} S_{\lambda_1}}} : \Jac(Y_2) \stackrel{\phi_2^*}{\lra} \Jac(Z) \stackrel{S_{\lambda_2}}{\lra} \Jac(Z) \stackrel{S_{\lambda_1}}{\lra} \Jac(Z) \stackrel{\Nm_{\phi_1}}{\lra} \Jac(Y_1) $$
which coincides with the definition of the morphism $\ol{S_{\lambda_1}S_{\lambda_2}}$. Thus $\widehat{S_{1,2}}=S_{2,1}$.
\end{proof}

\subsection{Calculation of the Schur  correspondence in the ring of endomorphisms}

\begin{prop} \label{k21k12} We have the relations $S_{2,1} S_{1,2} =[N] \in \End(P_{\lambda_1})$ and $S_{1,2} S_{2,1}=[N] \in \End(P_{\lambda_2})$ where $[N]$ denotes the multiplication by the integer $N= \ds{\frac{|W|^2 (\lambda_1, \lambda_1) (\lambda_2, \lambda_2)}{|H_1||H_2|{(\dim(V))}^2}}.$
\end{prop}
\begin{proof} We denote $\tilde{q}= \frac{\dim(V)}{|H_1||H_2||W|(\lambda_1, \lambda_2)} $ and $q = \frac{|W|}{\dim(V)}$. Let us omit $q$ as in the last proposition. We have 
$$S_{2,1} S_{1,2}: \Jac(Y_1) \sr{\phi^*_1}{\lra} \Jac(Z) \sr{S_{\lambda_1}}{\lra} \Jac(Z) \sr{S_{\lambda_2}}{\lra} $$
$$\Jac(Z) \sr{\Nm_{\phi_2}}{\lra} \Jac(Y_2) \sr{\phi^*_2}{\lra} \Jac(Z) \sr{S_{\lambda_2}}{\lra} \Jac(Z) \sr{S_{\lambda_1}}{\lra} \Jac(Z) \sr{\Nm_{\phi_1}}{\lra} \Jac(Y_1)$$
Now the composition of $\phi^*_2 \Nm_{\phi_2}$ restricted to the variety $P_{\lambda_2} \subset \Jac(Y_2) \subset \Jac(Z)$ coincides with the multiplication by $\deg(\phi_2)=|H_2|$. On the otherhand, according to \cite{merindol} Corollaire 3.3 (a) and (c), we have 
$S^2_{\lambda_2}= q (\lambda_2, \lambda_2) S_{\lambda_2}$ and $S_{\lambda_1}S_{\lambda_2}S_{\lambda_1}= q^2 (\lambda_1, \lambda_2)^2 S_{\lambda_1}$. Moreover, $S_{\lambda_1}$ operated on $P_{\lambda_1}$ as multiplication by $q(\lambda_1, \lambda_1)$ and $Nm_{\phi_1} \phi_1^*$ operates on $P_{\lambda_1}$ as multiplication by $\deg(\phi_1)= |H_1|$. In this way, multiplying the factors we get $S_{2,1}S_{1,2}$ operates on $P_{\lambda_1}$ as multiplication by $\frac{|W|^2 (\lambda_1, \lambda_1) (\lambda_2, \lambda_2)}{|H_1||H_2|{(\dim(V))}^2}$.
\end{proof}

\begin{rem} This proposition generalises for two weights the lemma 3.8 of \cite{lp} since for $\lambda_1 = \lambda_2$ and $S_{1,2}= \frac{1}{|H|^2} \ol{S_{\lambda}}=S_{2,1}$ and ${\ol{S_{\lambda}}}^2= \ol{e}S_{\lambda}$.
\end{rem}

\section{The distinguished form for Schur correspondences}

We start by an observation that the trace correspondence between two curves $Y_1$ and $Y_2$ induce the zero isogeny when restricted to the Prym variety $P_\la$.

\begin{defi} Let $\ol{T}$ denote the trace correspondence on $Y_1$ and $Y_2$, i.e given as the composition of $$\ol{T}: \Jac(Y_1) \sr{\Nm_{\phi_1}}{\lra} \Jac(X) \sr{\phi^*_2}{\lra} \Jac(Y_2)$$
$$\begin{matrix}
t: & Y_1 & \ra & \Jac(Y_2) \\
         & y_1 & \ms & \ds{\sum_{y_2 \in \psi^{-1}_2( \psi_1 (y_1))}} y_2
\end{matrix}$$
\end{defi}
We denote $\phi_{\ol{T}}$ the morphism from $\Jac(Y_1)$ to $\Jac(Y_2)$ induced by $\ol{T}$.

\begin{prop} \label{tracenullsurprym} We have $\ol{T}|_{\ol{P_{\lambda_i}}} =0$ for $i=1,2$.
\end{prop}
\begin{proof} We have $\ol{T} \ol{S_{\lambda_1}} = \ol{(\phi_2^* \Nm_{\phi_1})} \circ \ol{S_{\lambda_1}} = \ol{\phi^*_2 (\Nm_{\phi_1} S_{\lambda_1})}$. Now $$\Nm_{\phi_1} S_{\lambda_1} (z) = \pi(\ds{\sum_{g \in W}} (\lambda g, \lambda) z^g) = (\ds{\sum_{w \in W}} \lambda w, \lambda) \pi(z) = 0$$ because $\ds{\sum_{g \in W}} \lambda g$ is a $W-$invariant vector of the Cartan-subalgebra $\hfr^*$ and is thus $0$.
\end{proof}

\begin{rem} An immediate consequence of the last proposition is that for all  $\alpha \in \QQ$, the endomorphisms $S_{1,2} + \alpha \ol{T}$ and $ S_{1,2}$ induce the same isogeny in $\Hom_{\QQ}(\ol{P_{\lambda_1}}, \ol{P_{\lambda_2}})$. 
It is desirable to use integral correspondences. So we may try to find a $\alpha \in \QQ$ such that $$(\la_1 g , \lambda_2) + \alpha \in \ZZ \forall g \in W$$
since $(S_{1,2} + \alpha t) (y_1) = S_{1,2} (y_1) + \alpha t (y_1) = $
$$\ds{\sum_{i=1}^d} (\lambda_2 g_i, \lambda_1) \phi_2(z^{g_i}) + \alpha \ds{\sum_{i=1}^d} \phi_2(z^{g_i}) = \ds{\sum_{i=1}^d} [(\lambda_2 g_i, \lambda_1) + \alpha] \phi_2(z^{g_i}).$$ 
This happens if and only if $(\la_2g - \la_2 g', \la_1) \in \ZZ \forall g,g' \in W$. By the $W-$invariance of $(,)$, it suffices that $M_g = (\la_2 g - \la_2, \la_1) \in \ZZ$ since $M_g - M_{g'}= (\la_2 g - \la_2 g', \la_1)$. 
\end{rem}

This motivates the following definition.

\begin{defi} We say that $(,)_{\lambda_1, \lambda_2}$ is the distinguished bilinear form for a couple of weights $(\lambda_1, \lambda_2)$ if $$(\lambda_2 g - \lambda_2, \lambda_1) \in \ZZ \quad \forall \quad g \in W$$
and if $(,)'$ is another bilinear form satisfying the aforesaid property then $(,)'=k(,)$ for some $k \in \ZZ \setminus \{0 \}$. 
\end{defi} 

\begin{rem} We observe that this condition is independant of the order of $\lambda_1$ and $\lambda_2$ because $$(\lambda_2 g - \lambda_2, \lambda_1) = (\lambda_1 g^{-1} - \lambda_1, \lambda_2).$$
\end{rem}

\begin{rem} If we take $\lambda_1 = \lambda_2$, we find the condition of Kanev \cite{kanev}.
\end{rem}

\begin{rem} For this distinguished form $S_{1,2}$ and $S_{2,1}$ are isogenies between $\ol{P_{\lambda_1}}$ and $\ol{P_{\lambda_2}}$.
\end{rem}

\subsection{Geometrical construction of the isogeny  $\Delta_{1,2}$}
In \cite{kanev1}, Kanev constructs a geometrical and integral correspondence, both in the one and two weight case, with the key idea of defining the distinguished form $(\la_1 w , \la_2 w')$ by the intersection pairing. In the one weight case, this correspondence has been generalised to an arbitrary base curve in Lange-Pauly \cite{lp}. In the introduction of \cite{lk}, Lange and Kanev remark that this correspondence can immediately be generalised, in the two weight case also, to an arbitrary base curve. So the question of comparison of the isogenies between Prym varieties induced by the Schur correspondence $S_{1,2}$ defined algebraically and the Kanev Correspondence $\Delta_{1,2}$ defined geometrically arises. We shall show in theorem \ref{deltaklien} that they satisfy a relation. This generalises the theorem 3.5 of \cite{lp} which shows that they satisfy a relation in one weight case. This implies that these correspondences induce the same isogenies between the Prym varieties. So we obtain nothing new.

For the construction of the Kanev correspondence we refer to \cite{kanev1} section 3,4,6, \cite{lp} and \cite{lk}. Let us take up Kanev's notation and introduce it only so much as to prove the theorem \ref{deltaklien}. Let $L$ be a lattice endowed with a $W-$action and a fix a bilinear form $(|)$ on $L$. Kanev in \cite{kanev1} section 4 constructs a new lattice $N(L, \la)= L \oplus \ZZ$ associated to any element $\la \in L$. One extends the $W-$action and the bilinear form, denoted $(,)$, from $L$ to $N(L,\la)$. Then in section 6, he defines a bigger lattice $N$ such that for all weights $\la$ of $L$, there exists a $ W-$equivariant embedding of $N(L, \la)$ into $N$. The lattice $N$ can also be endowed with a $W-$action and a bilinear form $B$ which will imitate the intersection pairing and these extend the action and bilinear forms on $N(L,\la)$ (cf. Proposition 6.3 \cite{kanev1}). 

Using the notation of the preceding sections, let us recall the following commutative diagram

\begin{equation*}
\xymatrix{ & Z \ar[ld]_{\phi_1} \ar[rd]^{\phi_2} & \\
           Y_1 \ar[rd]_{\psi_1} &              & Y_2 \ar[ld]^{\psi_2} \\
           &              X              & 
}
\end{equation*}
Let $U$ be the open subset  $X \setminus$ $\{$ the ramification points of $\pi$, $\psi_1$ and $\psi_2 \}$. We fix a point $\xi_0 \in U$ and $z_0 \in Z$ such that $\pi(z_0)=\xi_0$. For $i=1,2$ we  denote by $\mu_i$ the $W-$equivariant bijections between the fibers of $\psi_i$ and the orbit of $\la_i$ with $\phi_i(z_0)$ as the point chosen in the fiber of $\psi_i(\xi_0)$. While in the one weight case, any point may be chosen in the fiber of $\psi$, in the two weight case a compatible choice of points must be made in the fiber of $\psi_1$ and $\psi_2$. So we choose rather a point in the fiber of $\pi$. As per \cite{kanev1} section 4.1,
 we denote by $\pi_i: N(\Lambda, \la_i)_{\QQ} \ra  L_{\QQ}$ the $W-$equivariant maps inducing $W-$equivariant bijections between the ``lines'' in $N(L,\la_i)$ and the orbit of weight $\la_i$ in $L_\QQ$.

Let us fix a bilinear form $(|)$ on $L_{\QQ}$ such that the weight lattice $P$ be contained in the dual lattice $L^*$. We denote by $b_1$ the symmetric bilinear form $$b_1: P/L \times P/L \ra \QQ/\ZZ.$$ We can find a lattice $K$ endowed with a symmetric integral bilinear form $B_2$ such that there exist $K_1$ satisfying the following conditions

\begin{enumerate}
\item $K \subset K_1 \subset K^*$ where $K^*$ is the dual lattice of $K$ with respect to $B_2$
\item there exists an isomorphism $\gamma: P/L \ra K_1/K$
\item for all $\ol{\la}, \ol{\mu} \in P/L$, we have 
$$ b_2(\gamma(\ol{\la}) , \gamma(\ol{\mu})) = - b_1( \ol{\la}, \ol{\mu})$$
where $b_2$ denotes the form induced upon  $K_1/K$ by $B_2$.
\end{enumerate}

In \cite{kanev} section 7.6, the construction of the lattice $K$ is explained. 
We shall admit the construction of such a lattice.

\begin{defi} We define $N \subset P \oplus K_1$ as
$$ N= \{ (\la, \eta) | \gamma(\la mod L) = \eta mod K \}$$
\end{defi}

\begin{defi} We define an action of $W$ on $P \oplus K_1$ as follows
$$\begin{matrix}
          W \times P \oplus K_1  & \ra & P \oplus K_1 \\
          (w, (\la, \eta))           & \ms & (\la w , \eta)
\end{matrix}$$
\end{defi}

\begin{defi} We denote by $B$ the symmetric bilinear form
$$\begin{matrix} 
        B : & (P \oplus K_1) \times (P \oplus K_1) & \ra & \QQ \\
            & ((\la_1, \eta_1), (\la_2, \eta_2))   & \ms & (\la_1|\la_2) + B_2(\eta_1,\eta_2)
\end{matrix}$$
\end{defi}

\begin{prop} [\cite{kanev1} Proposition 6.3] \label{propN} The following propositions hold
\begin{enumerate}
\item The lattice $N$ is invariant under $W-$action.
\item The bilinear form $B$ takes integral values on  $N$.
\item For $\la \in P$, there exists a lattice $N'(L, \la) \subset N$ isomorphic to $N(L, \la)$ as a $\ZZ[W]-$module.
\item There exists $m \in \ZZ$ such that the restriction of $B$ to $N'(L, \la)$ is equal to $(,)_m$.
\end{enumerate}
\end{prop}

By the proposition \ref{propN}, we embed $N(L, \la_1)$ and $N(L, \la_2)$ in $N$. Let $\{l_1, \cdots, l_d \}$ and $\{ l_1', \cdots , l_e' \}$ be the orbits of $(0,1) = l_1$ in $N(L, \la_1)$ et $(0,1)=l_1'$ in $N(L, \la_2)$ respectively.

\begin{defi} We define the correspondence $\Delta_{1,2}$ between the curves $Y_1$ and $Y_2$ by 
$$\begin{matrix} 
\Delta_{1,2}:  & Y_1 & \ra & \Div(Y_2) \\
              & y   & \ms & \ds{\sum^e_{j=1} B(\pi_1^{-1}(\mu_1(y)), l_j') \mu_2^{-1}(\pi_2(l_j'))}
\end{matrix}$$
\end{defi}
 
The verification that the definition of $\Delta_{1,2}$ is independant of the path and of the point $z \in Z$ chosen follows from the $W-$invariance of the bilinear form $B$ (cf. Lemma 3.2 \cite{lp}).

\begin{rem} The correspondence $\Delta_{1,2}$ is integral by the proposition \ref{propN}. It may not be effective. The correspondence $S_{1,2}$ is not necessarily integral.
\end{rem}

\begin{thm} \label{deltaklien} We have the equality $$\Delta_{1,2} = S_{1,2} + (B(l_1,l_1') - (\la_1, \la_2)) \ol{T}$$ where $\ol{T}$ is the trace correspondence between the curves $Y_1$ and $Y_2$.
\end{thm}
\begin{proof} It suffices to verify the relation for $y \in Y_1$ such that $\pi_1^{-1}(\mu_1(y))=l_1$. By the proof of the proposition \ref{propN} (3), we see that $l_1$ is of the form $(\la_1, \eta_1) \in N'(L, \la_1)$ and $l_1'$ is of the form $(\la_2, \eta_2) \in N'(L,\la_2)$. Thus $B(l_1, l'_j) = B(l_1, l'_1 w_j) = B((\la_1, \eta_1), (\la_2, \eta_2) w_j) = B((\la_1, \eta_1), (\la_2 w_j, \eta_2)) = (\la_1| \la_2 w_j) + B_2(\eta_1, \eta_2) = (\la_1| \la_2 w_j - \la_2) + (\la_1| \la_2) + B_2(\eta_1, \eta_2) = (\la_1| \la_2 w_j - \la_2) + B((\la_1, \eta_1), (\la_2, \eta_2)) = (\la_1|\la_2 w_j - \la_2) + B(l_1, l_1')$. Thus, we obtain 
$$ B(l_1, l_1' w_j) - (\la_1, \la_2 w_j) = B(l_1, l_1') - (\la_1, \la_2)$$
for all $w_j \in W$. Let us write this difference as $r \in \QQ$. We have $$\Delta_{1,2}(y)= \ds{\sum^e_{j=1} B(\pi_1^{-1}(\mu_1(y)), l_j') \mu_2^{-1} (\pi_2(l_j')) = \sum^e_{j=1} B(l_1, l_1' w_j) \mu_2^{-1}(\pi_2( l_1' w_j))}$$
$ = \sum^e_{j=1} (\la_1, \la_2 w_j) \mu_2^{-1}(\pi_2(l_1' w_j)) + r \ol{T} = k_{1,2}(y) + r \ol{T}$ by the $W-$equivariance of $\mu_2$ and $\pi_2$.
\end{proof}

\begin{cor} \label{sameiso} The correspondence $\Delta_{1,2}$ induces the same isogeny as the correspondence $S_{1,2}: P_{\la_1} \ra P_{\la_2}$.
\end{cor}
\begin{proof} The relation $\Delta_{1,2} = S_{1,2} + r \ol{T}$ implies this assertion because by the proposition \ref{tracenullsurprym} $\ol{T}$ acts trivially on
 $P_{\la_1}$. 
\end{proof}

\begin{cor} \label{isogeniedelta} We have $\Delta_{2,1} \circ \Delta_{1,2}= [N] \in \End(P_{\la_1})$ where $N$ is the integer  $$\frac{|W|^2 (\la_1, \la_1) (\la_2, \la_2)}{|H_1||H_2| \dim(V)^2}$$ of the proposition \ref{k21k12}.
\end{cor}
\begin{proof} We have $\Delta_{2,1} \circ \Delta_{1,2} = (S_{2,1} + r'\ol{T}) (S_{1,2} + r \ol{T}) = S_{2,1}S_{1,2} + r S_{2,1} \ol{T} + r' \ol{T} S_{1,2} + rr' \ol{T}^2$. We conclude by the propositions \ref{tracenullsurprym}, \ref{propk12} and \ref{k21k12}.
\end{proof}

\section{Comparison result}
\begin{defi} Let $A$ be an abelian variety. For a line bundle $L$ on $A$, we denote $\varphi_L : A \ra \wh{A}$ the map $a \ms T_a^*L \otimes L^{-1}$.
\end{defi}
\begin{defi} For a smooth projective curve $Y$, we denote the principal polarisation $$\alpha:\Jac(Y) \ra \wh{\Jac(Y)}.$$ 
\end{defi}

\begin{prop} \label{comppol} Let $Y_1$ and $Y_2$ be two smooth projective curves. Let $P_i \subset \Jac(Y_i), i=1,2$ be sub-abelian varieties. Let $\varphi_i$ be the restriction of the principal polarisation $\alpha_i$ on $ \Jac(Y_i)$ to $P_i$. Let $\phi: \Jac(Y_1) \ra \Jac(Y_2)$ be a  homomorphism such that $\phi$ restricted to $P_1$, denoted $\ol{\phi}$, be an isogeny with values in $P_2$. We denote by $\wh{\phi}: \Jac(Y_2) \ra \Jac(Y_1)$ the composition of 
$$\Jac(Y_2) \sr{\alpha_2}{\lra} \wh{\Jac(Y_2)} \sr{\wh{\phi}}{\lra} \wh{\Jac(Y_1)} \sr{\alpha_1^{-1}}{\lra} \Jac(Y_1)$$ 
and by $\ol{\wh{\phi}}$ the restriction of $\wh{\phi}$ to $P_2$.  Then $\ol{\wh{\phi}}$ takes values in $P_1$.  Suppose that there exists an integer $N$ such that $\ol{\wh{\phi}} \ol{\phi}$ be multiplication by $N$ in $P_1$, then we have the equalities $$ \ol{\phi}^* \varphi_{2} = [N] \varphi_{1}, \qquad \ol{\wh{\phi}}^* \varphi_1 = [N] \varphi_{2}.$$
\end{prop}  

\begin{proof} The map $\ol{\wh{\phi}}$ takes values in $P_1$ because it is the isogeny dual to $\ol{\phi}$.
We have the following commutative diagram
\begin{equation*}
\xymatrix{
\Jac(Y_1) \ar[rrr]^{\phi} \ar[ddd]^{\alpha_1} &             &          &  \Jac(Y_2) \ar[ddd]^{\alpha_2} \\
                                    & P_1 \ar[lu]^{i_{P_1}} \ar[r]^{\ol{\phi}} & P_2 \ar[ru]^{i_{P_2}} \ar[d]^{\varphi_2} & \\
                                    & \wh{P_1}    & \wh{P_2} \ar[l]^{\wh{\ol{\phi}}}  &     \\
\wh{\Jac(Y_1)} \ar[ru]^{\wh{i_{P_1}}} &             &          & \wh{\Jac(Y_2)} \ar[lll]^{\wh{\phi}} \ar[lu]^{\wh{i_{P_2}}}
}
\end{equation*}

We have $\ol{\phi}^* \varphi_2 = \ol{\wh{\phi}} \varphi_2 \ol{\phi}$. Now $\varphi_2= \wh{i_{P_2}} \alpha_2 i_{P_2}$, $i_{P_2} \ol{\phi} = \phi i_{P_1}$ and $\ol{\wh{\phi}} \wh{i_{P_2}} = \wh{i_{P_1}} \wh{\phi}$. So we obtain ${\ol{\phi}^* \varphi_2}= \wh{i_{P_1}} \wh{\phi} \alpha_2 \phi i_{P_1}$. Now one has  $[N] = \ol{\wh{\phi}} \ol{\phi} = \alpha_1^{-1} \wh{\phi} \alpha_2 \phi i_{P_1}$ from which we conclude that $[N]\alpha_1 = \wh{\phi} \alpha_2 \phi i_{P_1}$. Thus ${\ol{\phi}^* \varphi_2}= [N] \wh{i_{P_1}} \alpha_1 i_{P_1} = [N] \varphi_1$. The other relation is proved similarly.
  
\end{proof}

\begin{cor} \label{compresult} We have $\Delta_{\la, \mu}^* \varphi_\mu = [N] \varphi_\la$ and $\Delta_{\mu, \la}^* \varphi_\la = [N] \varphi_\mu$.
\end{cor}
\begin{proof} This follows by combining proposition \ref{comppol} and corollary \ref{isogeniedelta}.
\end{proof}

\section{Application to the calculation of polarisation}
In this section we calculate the integer $N= N_{\la_1, \la_2} = \frac{|W|^2 (\lambda_1, \lambda_1)_{\lambda_1, \lambda_2} (\lambda_2, \lambda_2)_{\lambda_1, \lambda_2}}{|H_1||H_2|(\dim(V))^2}$ for various couples of fundamental weights.

\begin{ex} 

$A_n$: We have $(,)_{\varpi_i, \varpi_j}=(,)_{C-K}$. We have $|W|= (n+1)!$, $|H_1|= k! (n+1-k)!$, $|H_2|= l! (n+1-l)!$, $\dim(V)=n$, $(\lambda_1, \lambda_1)= \frac{-k (n+1-k)}{n+1}$ and $(\lambda_2, \lambda_2)=\frac{-l (n+1-l)}{n+1}$. We deduce that$N= C^{k-1}_{n-1} C^{l-1}_{n-1}$. We remark that when $k=1$ and $l=n$, we have $N=1$. Thus the prym varieties $P_{\varpi_1}$ and $P_{\varpi_n}$ are isomorphic.
\end{ex}

\begin{ex} $B_l$. We have $(,)_{\varpi_i, \varpi_j}= (,)_{C-K}$ for $i,j<l$ and $(,)_{\varpi_i, \varpi_l}=2 (,)_{C-K}$. We have $|W|=2^l l!$. We have $|H_{\varpi_i}|=i! (l-i)! 2^{l-i}$ for $i < l$ and $|H_{\varpi_l}|=l!$. We have $(\varpi_i, \varpi_i)_{\varpi_i, \varpi_j}=i$ for $i < l$ and $(\varpi_l, \varpi_l)_{\varpi_i, \varpi_l}=l/2 \times 2$. Thus we get that 
$$N_{\varpi_i,\vp_j}=  \left\{ 
\begin{array}{rcl} 2^{i+j} C^{i-1}_{l-1} C^{j-1}_{l-1} & i,j < l \\
         2^{i+l+2} C^{i-1}_{l-1} & i, j=l 
\end{array}
 \right.$$
\end{ex}

\begin{ex} $C_l$. We have $(,)_{\varpi_i, \varpi_j}= (,)_{C-K}$. We have $|W|=2^l l!$. We have $(\varpi_i, \varpi_i)=i$ for all $i$. We have $|H_{\varpi_i}|= i! (l-i)! 2^{l-i}$. Thus we get that $$N_{\varpi_i,\varpi_j}= 2^{i+j} C^{i-1}_{l-1} C^{j-1}_{l-1}.$$
\end{ex}

\begin{ex} $D_l$ We have $(,)_{\varpi_i, \varpi_j}=(,)_{C-K}$ for all $i,j$. We have $|W|= l! 2^{l-1}$. We have $|H_{\varpi_i}|= |i! (l-i)! 2^{n-i-1}$ for $i \leq l-2$, $|H_{\varpi_{l-1}}|=(l-1)!$ and $|H_{\varpi_l}|=l!$. We have $(\varpi_i, \varpi_i)=i$ for $i \leq l-2$ and $(\varpi_{l-1},\varpi_{l-1})= (\varpi_l, \varpi_l)=l/4$. Thus we get that 

$$N_{\varpi_i,\vp_j} = \left\{ 
\begin{array}{rcl}
 2^{i+j} C^{i-1}_{l-1} C^{j-1}_{l-1} & i,j \leq l-2 \\ 
 2^{l-3+i} i C^i_l                   & i \leq l-2, j = l-1   \\
 2^{l-3+i}C^{i-1}_{l-1}              & i \leq l-2 , j=l   \\
 2^{2(l-3)}l                         & i=l-1, j=l
\end{array}
\right.$$
\end{ex}

\section{Applications to Abelianisation}
In this section we consider a simple Lie Group $G$ of type $A,D,E$.

\begin{prop} [Lange-Pauly] \cite{lp} Let $Z$ denote a smooth projective curve admitting a free action by a group $W$. Let $T$ denote a torus in an algebraic group $G$ and $\la$ a weight. The canonical evalution map $\ev_\la: H^1(Z, \ul{T})^W \ra \Jac(Z)$ from $W-$invariant $T-$bundles on $Z$ to line bundles on $Z$ lifts to $\Jac(Z/\Stab(\la))$ i.e the line bundles in the image can be endowed with a canonical $\Stab(\la)-$linearisation.
\end{prop}
The lift is denoted by $\tilde{\ev_\la}$.
\begin{prop} \label{pdonpla} Consider the diagram
\begin{equation*}
\xymatrix{
H^1(Z,\ul{T})^W \ar[r]^{~~~\evt_\la} \ar[rd]_{\evt_\mu} & P_\la \ar[d]^{\Delta_{\la, \mu}}\ar[r] & \Jac(Y_\la) \ar[d]^{\Delta_{\la, \mu}}\\
                                         & P_\mu \ar[r] & \Jac(Y_\mu)
}
\end{equation*}
If $(\la,\mu) \neq 0$ then we have the equality
$$ \Delta_{\la, \mu} \tilde{\ev_\la} = {\frac{|W|}{\dim(V)}} \frac{(\la, \la) (\mu, \mu)}{|H_\la||H_\mu|} \tilde{\ev_\mu}.$$
\end{prop}
\begin{proof} Put $q=\frac{|W|}{\dim(V)^2}$. By proposition \ref{compcorbasepoint} and corollary \ref{sameiso} , denoting by $S_{\la,\mu}$ and $\Delta_{\la,\mu}$ the isogenies that these correspondences induce, we have 
$$\ol{S_\mu S_\la} \frac{\dim(V)^2}{|W|(\la, \mu) |H_\la| |H_\mu|}= S_{\la, \mu} =\Delta_{\la, \mu}.$$
For ease of notation, let us calculate rather with $\ol{S_\mu S_\la}$. We have
$$ \ol{S_\mu S_\la} (E \times_{\la} \CC^*) = q (\la, \mu) (E \times_\la \CC^*) (\sum_{s \in W} (\la s, \mu) s) = q (\la, \mu) E \times_{\sum_{w \in W} (\la s, \mu) \la s} \CC^*.$$

We denote $\gamma = \sum_{s \in W} (\la s, \mu) \la s$.
Now, for any weight $\kappa$, let us calculate  $ (\gamma, \kappa)$. We have
$$(\gamma, \kappa) = \sum_{s \in W} (\la s, \mu) (\la s, \kappa)$$
which by lemma \ref{sumcal} is equal to 
$$q (\la, \la) (\mu, \kappa).$$
Thus extending $\mu$ to a basis of $V$ with other vectors orthogonal to $\mu$, we get that
$$\gamma = q (\la, \la) (\mu, \mu) \mu.$$

Thus $\ol{S_\mu S_\la} (E \times_{\la} \CC^*) = q^2 (\la,\mu) (\la, \la) (\mu, \mu) E \times_\mu \CC^*$ which is equivalent to the assertion in the proposition.
\end{proof}

\begin{prop} $\Delta_{2,3} \Delta_{1,2} = \frac{|W| (\la_2, \la_2)^2}{\dim(V)^2 |H_2|^2} \Delta_{1,3}$.
\end{prop}
\begin{proof} The map $\evt_\la$ is surjective onto the Prym variety $P_\la$. So it suffices to take a $T-$bundle $E \in H^1(Z,\ul{T})^W$ and compare $\Delta_{2,3} \Delta_{1,2} \evt_{\la_1}(E)$ with $\evt_{\la_3} (E)$. By proposition \ref{pdonpla} we get the desired result.
\end{proof}

Let $P(G)$ denote the weight lattice of $G$. We have a canonical bijection  
$$H^1(Z,\ul{T}) \ra \Hom(P(G), \Pic(Z)).$$
Now the group $W$ acts on the curve $Z$ and therefore on line bundles on $Z$. It also acts on the weight lattice $P(G)$. The $W-$equivariant homomorphisms correspond to $W-$invariant $T-$bundles on $Z$ for the twisted action.

\begin{prop} \label{invevt} Let $M$ denote the exponent of the group $P(G)/\ZZ[W] \la$. We have an inverse isogeny $\delta_\la$ to $\evt_\la$ such that $$\delta \evt_\la= [M]$$ in $H^1(Z,\ul{T})^W$.
\end{prop}
\begin{proof} Let $L \in P_\la$. We define $\psi_L :  \ZZ[W] \la  \ra  \Pic(Z)$ as the unique $W-$equivariant map sending $\la$ to $L$. Now identifying $H^1(Z, \ul{T})^W$ with $\Hom_W(P(G), \Pic(Z))$ we get $\delta$ as the composition of $\psi_L [M]$. 
\end{proof}

\begin{cor} We have an isogeny between the Prym varieties $ \evt_\mu \delta_\la: P_\la \ra P_\mu$. Let $M_\la$ (resp. $M_\mu$) denote the exponent of $P(G)/\ZZ[W]\la$ (resp. $P(G)/\ZZ[W]\mu$). We have $$\evt_\la \delta_\mu \evt_\mu \delta_\la = [M_\mu M_\la].$$
\end{cor}
\begin{proof} It is immediate from the proposition \ref{invevt}.
\end{proof}

\begin{cor} Consider the diagram
$$\begin{matrix}
\xymatrix{
H^1(Z,\ul{T})^W & P_\la \ar[l]^{\delta_\la} \ar[d]^{\Delta_{\la, \mu}} \\
                & P_\mu \ar[ul]^{\delta_\mu}
}
\end{matrix}$$
We have the equality
$$ \delta_\mu \Delta_{\la, \mu} = \frac{|W| (\la, \la)(\mu, \mu) M_\mu}{\dim(V) |H_\la||H_\mu||M_\la|} \delta_\la.$$
\end{cor}
\begin{proof} Calculating $\delta_\mu (\Delta_{\la, \mu} \evt_\la) \delta_\la$ by proposition \ref{pdonpla} and applying proposition \ref{invevt} we get the desired result.
\end{proof}

%\begin{rem} The hypothesis that $\ZZ[W]=\Lambda$ in the main theorem of [Lange-Pauly] \cite{lp} can be dropped in the light of proposition \ref{invevt}. Since we needed the equality only to pull the pull-backs of $L_{\SL(V_\la)}$ and $L_G$ to $P_\la$.
%\end{rem}

\section{Open problems}
In view of proposition \ref{polonisotypcompsplits} and \ref{polonPlasplits}, one would like to calculate the translation subgroups of the restriction of the principal polarisations on $\Jac(Z)$ to the Prym varieties $P_V$ and $P_\la$. One would also like to calculate the Mumford Theta groups of a line bundle representing these restricted polarisations and the Weil pairing induced on its associated translation subgroups.

One would like to compute the kernel of $\evt_\mu: H^1(Z,\ul{T})^W_{\eta} \ra P_\mu$ of the proposition \ref{pdonpla} for $\mu$ arbitrary. The reason is that one could then compute the restriction of the principal polarisation $\varphi_\mu$ on $\Jac(Y_\mu)$ to $P_\mu$ as follows. Let $\Lambda$ denote the lattice associated to an irreducible representation defined over $\QQ$ of a Weyl group $W$. Then there are elements $\la \in \Lambda$ such that $\evt_\la$ be an isomorphism. This can be seen as follows. One has $\ker(\evt_\la) \subset \ker(ev_\la)$ where
$$\begin{matrix}
\ev_\la: & H^1(Z,\ul{T})^W & \ra & \Jac(Z) \\
         & E               & \ms & E\times_\la \CC^*.
\end{matrix}$$
and $T= \Hom(\Lambda, \CC^*)$.
By Lange-Pauly \cite{lp} proposition 5.2 one has 
$$\ker(ev_\la)= \Hom_W(\Lambda/\ZZ[W]\la, \Jac(Z))$$
and one has many weights $\la$ for which $\ZZ[W]\la=\Lambda$. (A list is given in \cite{lp} Theorem 8.1). Then one can calculate the kernel of $\Delta_{\la, \mu}: P_\la \ra P_\mu$ by proposition \ref{pdonpla}. Now by the general theory \cite{mumford} Chapter 23, for the inclusion
$$\{0\} \subset \ker(\Delta_{\la, \mu}) \subset \ker(\Delta_{\la, \mu})^{\perp} \subset K(\Delta^*_{\la, \mu} (\varphi_\mu |_{P_\mu})) $$
one obtains $$\ker(\Delta_{\la, \mu})^{\perp}/\ker(\Delta_{\la, \mu}) = K(\varphi_u |_{P_u})$$ and by corollary \ref{compresult} one can compute $K(\Delta^*_{\la, \mu} (\varphi_\mu |_{P_\mu})) $. 

Recall that the proof of the computation of the dimension of Verlinde Spaces in [BNR]\cite{bnr} for $\SL(n)$ consisted of two parts - namely, the abelianisation part and the computation of the restriction of the principal polarisation on the Jacobian of the spectral curve $X_s$ to the Prym variety $\Prym(X_s/X)$. So for a general reductive group $G$, it seems interesting to compute $ K(\varphi_u |_{P_u})$ for all weights $\mu$ to extend this proof technique. 

One can describe the kernel of $\evt_\mu$ geometrically as follows. Recall by Proposition 6.6 \cite{lp}, for $E_T \in H^1(Z,\ul{T})^W$ the bundle $E_T \times_T N$ admits a canonical $W-$linearisation. Let $\CC_\mu$ denote the one dimensional representation upon which $T$ acts by character $\mu$. We denote by $V(\la)$ the isotypical component of $\la$ in $\Res \Ind_T^N(\CC_\mu)$. The proof of Proposition 6.11 in \cite{lp} shows that $E_N \times V(\la)$ admits a $\Stab(\mu)$ linearisation. Now the linearisations on $\ev_\mu(E)$ and $E_N \times V(\la)$ do not necessarily preserve the natural inclusion $\ev_\mu(E) \hra E_N \times V(\la)$. The kernel of $\evt_\mu$ consists of those $T-$bundles in $\ker(\ev_\mu)$ such that the linearisations respect the inclusion.

\end{document}